\documentclass[12pt]{amsart}
\usepackage{amsmath}
\usepackage{hyperref}
\usepackage{amssymb}
\usepackage{latexsym}
\usepackage{amscd}
\usepackage{graphicx} 
\usepackage{amsthm}
\usepackage{mathrsfs}
\usepackage{xypic}
\usepackage{bm}

\newdimen\AAdi%
\newbox\AAbo%
\def\AAk#1#2{\s_etbox\AAbo=\hbox{#2}\AAdi=\wd\AAbo\kern#1\AAdi{}}%
\def\AAr#1#2#3{\s_etbox\AAbo=\hbox{#2}\AAdi=\ht\AAbo\raise#1\AAdi\hbox{#3}}%
\font\tenmsb=msbm10 at 12pt \font\sevenmsb=msbm7 at 8pt
\font\fivemsb=msbm5 at 6pt
\newfam\msbfam
\textfont\msbfam=\tenmsb \scriptfont\msbfam=\sevenmsb
\scriptscriptfont\msbfam=\fivemsb

\newtheorem{thm}{Theorem}[section]
\newtheorem{lem}{Lemma}[section]

\newcommand{\ba}{\begin{array}}
\newcommand{\ea}{\end{array}}

\newcommand{\Section}[2]{\setcounter{equation}{0}
\allowdisplaybreaks
\section[#1]{#2}}

\def\n{\nabla}

\def\sr#1{\mathscr{#1}}

\def\f#1#2{\frac{#1}{#2}}

\def\a{\alpha}

\def\de{\delta}
\def\De{\Delta}

\def\ep{\varepsilon}

\def\la{\lambda}

\def\th{\theta}

\begin{document}
\title
[Some results on Chern's problem] {Some results on Chern's problem}

\author
[Qi Ding and Y.L. Xin ]{Qi Ding and Y. L. Xin }
\address
{Institute of Mathematics, Fudan University, Shanghai 200433,
China}\email{09110180013@fudan.edu.cn}  \email{ylxin@fudan.edu.cn}
\thanks{The research was partially supported by
NSFC and SFECC}
\begin{abstract}
For a compact minimal hypersurface $M$ in $S^{n+1}$ with the
squared length of the second fundamental form $S$ we confirm that
there exists a positive constant $\de(n)$ depending only on $n,$
such that if $n\leq S\leq n +\delta(n)$, then $S\equiv n$, i.e.,
$M$ is a   Clifford minimal hypersurface, in particular, when
$n\ge 6,$  the pinching constant $\de(n)=\f{n}{23}.$
\end{abstract}

\renewcommand{\subjclassname}{
 \textup{2000} Mathematics Subject Classification} 
\date{}

\maketitle

\Section{Introduction}{Introduction}
\medskip

Let $M$ be a compact (without boundary) minimal hypersurface in
the unit sphere $S^{n+1}$ with the second fundamental form $B,$
which can be viewed as a cross-section of the vector bundle
Hom($\odot^2TM, NM$) over $M,$ where  $TM$ and $NM$ denote the
tangent bundle and the normal bundle  along $M$, respectively.
Minimal submanifolds in the sphere are  interesting not only in
its own right, but also are related to other interesting problems
(see \cite{x}, for example). The simplest hypersurface in
$S^{n+1}$ is the $n-$equator, totally geodesic hypersurface. The
important examples of minimal hypersurfaces in $S^{n+1}$ are
Clifford minimal hypersurfaces
$$S^k\left(\sqrt{\frac{k}{n}}\right)\times
S^{n-k}\left(\sqrt{\frac{n-k}{n}}\right),\quad k = 1, 2, \cdots,
n-1.$$

J. Simons \cite{si} discovered the intrinsic rigidity result.
Shortly afterwards  Chern-do Carmo-Kobayashi \cite{c-d-k} and
Lawson \cite{l} independently showed that Simons' result is sharp
and the equality is realized by the Clifford minimal
hypersurfaces in the unit sphere.  In their same paper
\cite{c-d-k},  Chern-do Carmo-Kobayashi proposed to study
subsequent gaps for the scalar curvature (or the squared length
of the second fundamental form $S=|B|^2$). The problem was also
collected by S. T. Yau  in the well-known problem section in
\cite{y1} and \cite{y2}.

Peng-Terng \cite{p-t1} made the first effort to attack the
Chern's problem and  confirmed the second gap. Precisely, they
proved that if the scalar curvature of $M$ is a constant, then
there exists a positive constant $C(n)$ depending only on $n$
such that if $n\leq S \leq n+C(n)$, then $S= n$, where $S$ stands
for squared norm of the second fundamental form. Later, the
pinching constant $C(n)$ was improved to $\f{1}{3}n,\, n>3$ by
Cheng-Yang \cite{c-y}, and  to $\f{3}{7}n,\, n>3$ by Suh-Yang
\cite{s-y}, respectively.

More generally, Peng-Terng \cite{p-t2} obtained  pinching results
for minimal hypersurfaces without the constant scalar curvature
assumption. They obtained that if $M$ is a compact minimal
hypersurface in $S^{n+1}$, then there exists a positive constant
$\delta(n)$ depending only on $n$ such that if $n\leq S \leq n +
\delta(n),\; n\le 5$, then $S \equiv n$ which characterize the
Clifford minimal hypersurfaces. Later, Cheng-Ishikawa \cite{c-i}
improved the the previous pinching constant  when $n\leq 5$, and
Wei-Xu \cite{w-x} extended the result to $n = 6, 7$. Recently,
Zhang \cite{z} extended  the  results to  $n\leq 8$ and improved
the previous pinching constant. The key point is to estimate the
upper bound of $\sr{A}-2\sr{B}$ in terms of $S$ and $|\n B|$ in
all the above mentioned papers (please see the definition of
$\sr{A},\sr{B}$ in next section).

In this paper, we continue to study the second gap problem
without the constancy of the scalar curvature. We obtain new
estimates for $\sr{A}-2\sr{B}$ in terms $S,\; |\n B|$ and another
higher order invariant of the second fundamental form of the
minimal hypersurface $M$ in $S^{n+1}$. Then, by the integral
formulas established in \cite{p-t2} we can carried out more
delicate integral estimates, which enable us to confirm the
second gap in all dimensions. Firstly, we give the  quantitative
result to show our technique fits all dimensions. Then, we refine
the estimates to obtain concrete pinching constant for dimension
$n\ge 6$ where they are better than all previous results for dimension $n\ge 7.$

\begin{thm}\label{th1}
Let M be a compact minimal hypersurface in $S^{n+1}$ with the
squared length of the second fundamental form $S.$ Then there
exists a positive constant $\de(n)$ depending only on $n$, such
that if $n\leq S\leq n +\delta(n)$, then $S\equiv n$, i.e., M is
a Clifford minimal hypersurface.
\end{thm}

\begin{thm}\label{th3}
If the dimension is $n\ge 6,$ then the pinching constant
$\de(n)=\f{n}{23}.$
\end{thm}

\bigskip\bigskip

\Section{Preliminaries}{Preliminaries}
\medskip

Let $M$ be a  minimal hypersurface in $S^{n+1}$ with the second
fundamental form $B$. We choose a local orthonormal frame field
$\{e_1,\cdots, e_n,\nu\}$ of $S^{n+1}$ along $M$,  such that
$e_i$ are tangent to $M$ and $\nu$ is normal to $M$.

Set $B_{e_i e_j}=h_{ij}\nu$. Then the coefficients of the second
fundamental form $h_{ij}$ is a symmetric $2-$tensor on $M$. Its
trace vanishes everywhere by the minimal assumption on the
submanifold $M$. Let  $S$ denote the squared length of the second
fundamental form of $M$
$$S=|B|^2=\sum_{i, j}h_{ij}^2.$$ By the Gauss equation it is an
intrinsic invariant related to the scalar curvature of $M$.

J. Simons obtained the following Bochner type formula \cite{si}
\begin{equation}\label{bs}
\frac{1}{2}\Delta S=|\nabla B|^2+S(n-S),
\end{equation}
where
$$|\nabla B|^2=\sum_{i,j,k}h_{ijk}^2,$$
$h_{ijk}$ is symmetric in $i,j$ and $k$ by the Codazzi equations. To
study the second gap problem Peng-Terng \cite{p-t2} computed the
second Bochner type formula as follows
\begin{equation}\label{bpt}
\f{1}{2}\De |\n B|^2=|\n^2 B|^2+(2n+3-S)|\n
B|^2+3(2\sr{B}-\sr{A})-\f{3}{2}|\n S|^2,\end{equation} where
$$|\n^2 B|^2=\sum_{i,j,k,l}h_{ijkl}^2,\quad
\sr{A}=\sum_{i,j,k,l,m}h_{ijk}h_{ijl}h_{km}h_{ml},\quad
\sr{B}=\sum_{i,j,k,l,m}h_{ijk}h_{klm}h_{im}h_{jl}.$$ It follows that
\begin{equation}\label{pt1}
\int_M\sum_{i,j,k,l}h_{ijkl}^2=\int_M\left((S-2n-3)|\nabla
B|^2+3(\sr{A}-2\sr{B})+\frac{3}{2}|\nabla S|^2\right).
\end{equation}

For any fixed point $x\in M$, we take orthonormal frame field
near $x$, such that $h_{ij} =\lambda_i\delta_{ij}$ at $x$ for all
$i, j$. Then
$$\sum_i\lambda_i=0,\quad\quad \sum_i\lambda_i^2=S$$ and
$$\sr{A}=\sum_{i,j,k}h_{ijk}^2\lambda_i^2,\quad
\sr{B}=\sum_{i,j,k}h_{ijk}^2\lambda_i\lambda_j.$$ There are
pointwise estimates
\begin{equation}\aligned\label{pt2}
&\lambda_j^2-4\lambda_i\lambda_j\leq\a S,\qquad
\forall i,j\\
&3(\sr{A}-2\sr{B})\le \a S|\n B|^2\endaligned
\end{equation}
with $\a=\frac{\sqrt{17}+1}{2}$ in \cite{p-t2} and
\begin{equation}\label{ci}\aligned
\sum_{i,j,k,l}h_{ijkl}^2\geq&\frac{3}{4}\sum_{i,j}(\lambda_i-\lambda_j)^2(1+\lambda_i\lambda_j)^2
+\frac{3S(S-n)^2}{2(n+4)}\\
=&\frac{3}{2}(Sf_{ 4}-f_{ 3}^{
2}-S^2-S(S-n))+\frac{3S(S-n)^2}{2(n+4)}
\endaligned
\end{equation}
in \cite{c-i}.  There is an integral equality
\begin{equation}\label{pt3}
\int_M(\sr{A}-2\sr{B})=\int_M(Sf_{ 4}-f_{ 3}^{
2}-S^2-\frac{1}{4}|\nabla S|^2),
\end{equation}
where $f_{ 3}=\sum_i\lambda_i^3$, $f_{ 4}=\sum_i\lambda_i^4$ in
\cite{p-t2}. In a general local orthonormal frame field
$f_3=\sum_{i, j, k}h_{ij}h_{jk}h_{ki}$ and $f_4=\sum_{i, j,
k,l}h_{ij}h_{jk}h_{kl}h_{li}.$

\bigskip\bigskip

\Section{New Estimates of $\sr{A}-2\sr{B}$}{New Estimates of
$\sr{A}-2\sr{B}$}

\medskip

In this paper we always assume  $S\ge n.$

Define
$$\sr{F}=\sum_{i,j}(\lambda_i-\lambda_j)^2(1+\lambda_i\lambda_j)^2,$$
then
$$\sr{F}=2(Sf_{4}-f_{ 3}^{ 2}-S^2-S(S-n)).$$
It is a higher order invariant of the second fundamental form.
\begin{lem}\label{ml}When the dimension $n\geq4$,
$$3(\sr{A}-2\sr{B})\leq 2S|\nabla
B|^2+C_1(n)|\nabla B|^2\sr{F}^{\f{1}{3}},$$ where
$C_1(n)=(\sqrt{17}-3){(6(\sqrt{17}+1))^{-\f{1}{3}}}(\frac{2}{\sqrt{17}}
-\frac{\sqrt{2}}{17}-\frac{1}{n})^{-\f{2}{3}}$.
\end{lem}
\begin{proof}
If there exist $i\neq j$ such that
$\lambda_j^2-4\lambda_i\lambda_j=tS>2S$, then by
$$S\geq\lambda_i^2+\lambda_j^2=\left(\frac{tS-\lambda_j^2}{4\lambda_j}\right)^2+\lambda_j^2,$$
we have $$\lambda_j^2\leq\frac{1}{17}(t+8+4\sqrt{4+t-t^2})S,$$
moreover,
\begin{equation}\label{l.1}
-\lambda_i\lambda_j\geq\frac{1}{4}\left(t-\frac{1}{17}(t+8+4\sqrt{4+t-t^2})\right)S=\frac{1}{17}(4t-2-\sqrt{4+t-t^2}\
)S.
\end{equation}
On the other hand,
\begin{equation}\label{l.2}
(\lambda_i-\lambda_j)^2=\frac{3}{4}\lambda_j^2-2\lambda_i\lambda_j+\lambda_i^2+\frac{1}{4}\lambda_j^2
\geq\frac{3}{4}\lambda_j^2-3\lambda_i\lambda_j= \frac{3t}{4}S.
\end{equation}
By the assumptions $n\ge 4$ and $S\geq n,$ and (\ref{l.1}) implies
$-\lambda_i\lambda_j\geq0.26S$, then combining (\ref{l.1}) and
(\ref{l.2}), we obtain
\begin{equation}\label{l.3}\aligned
\sr{F}=\sum_{k,l}&(\lambda_k-\lambda_l)^2(1+\lambda_k\lambda_l)^2\\
&\geq 2 (\lambda_i-\lambda_j)^2(1+\lambda_i\lambda_j)^2
\geq \frac{3t}{2}S(1+\lambda_i\lambda_j)^2\\
&\geq\frac{3t}{2}\left(-\lambda_i\lambda_j-\frac{S}{n}\right)^2S\geq
\frac{3t}{2}\left(\frac{1}{17}(4t-2-\sqrt{4+t-t^2})-\frac{1}{n}\right)^2S^3.
\endaligned
\end{equation}
Define a function
\begin{equation}\label{l.4}
\zeta(t)\triangleq
\frac{t}{(t-2)^3}\left(\frac{1}{17}(4t-2-\sqrt{4+t-t^2})-\frac{1}{n}\right)^2
\end{equation} on the interval $(2,\frac{\sqrt{17}+1}{2}].$ Then we have following rough estimate,
\begin{equation}\label{l.5}\aligned
\min_{(2,\frac{\sqrt{17}+1}{2}]}\zeta(t)&\ge
\min_{(2,\frac{\sqrt{17}+1}{2}]}\frac{t}{(t-2)^3}\left(\frac{1}{17}(4t-2-\sqrt{2})-\frac{1}{n}\right)^2\\
&=4\frac{\sqrt{17}+1}{(\sqrt{17}-3)^3}\left(\frac{2}{\sqrt{17}}-\frac{\sqrt{2}}{17}-\frac{1}{n}\right)^2.
\endaligned\end{equation}
From (\ref{l.3}), (\ref{l.4}) and (\ref{l.5}) we obtain
\begin{equation}\label{l.6}\aligned (\la_j^2-4\la_i\la_j&-2S)^3 =(t-2)^3S^3
 \le\f{2\sr{F}}{3\zeta(t)}\\
 & \le \frac{(\sqrt{17}-3)^3}{6(\sqrt{17}+1)}
 \left(\frac{2}{\sqrt{17}}-\frac{\sqrt{2}}{17}-\frac{1}{n}\right)^{-2}\sr{F}
 \triangleq (C_1(n)\sr{F}^{1/3})^3,
 \endaligned\end{equation}
where
$C_1(n)=(\sqrt{17}-3){(6(\sqrt{17}+1))^{-\f{1}{3}}}(\frac{2}{\sqrt{17}}-\frac{\sqrt{2}}{17}-\frac{1}{n})^{-\f{2}{3}}.$
By the definition of $\sr{A}$ and $\sr{B}$ and (\ref{l.6}), we
have
\begin{equation}\label{l.7}\aligned
3(\sr{A}-2\sr{B})=&\sum_{i,j,k}h_{ijk}^2(\lambda_i^2+\lambda_j^2+\lambda_k^2-2\lambda_i\lambda_j
-2\lambda_j\lambda_k-2\lambda_i\lambda_k)\\
\leq&\sum_{i,j,k\ distinct}h_{ijk}^2(2(\lambda_i^2+\lambda_j^2+\lambda_k^2)-(\lambda_i+\lambda_j+\lambda_k)^2)\\
&\qquad +3\sum_{j,i\neq j}h_{iij}^2(\lambda_j^2-4\lambda_i\lambda_j)\\
\leq&2S\sum_{i,j,k\ distinct}h_{ijk}^2+3\sum_{i\neq
j}h_{iij}^2(2S+C_1(n)\sr{F}^{1/3})\\
\le& 2S|\n B|^2+ C_1(n)\sr{F}^{1/3}.
\endaligned
\end{equation}
The lemma holds obviously when
$\lambda_j^2-4\lambda_i\lambda_j\le 2 S$ for any $i$ and $j.$
\end{proof}

The following estimates are applicable for higher dimension.

\begin{lem}\label{rl4}
If $n\geq6$ and $n\leq S\leq\frac{16}{15}n,$ then
$$3(\sr{A}-2\sr{B})\leq(S+4)|\nabla B|^2+C_3(n)|\nabla B|^2\sr{F}^{\f{1}{3}}$$ with
$$C_3(n)=\left(\frac{3-\sqrt{6}-4p}{\sqrt{6}-1+13p}(6-\sqrt{6}-13p)^2\right)^{\f{1}{3}},\;
p=\frac{1}{13(n-2)}.$$
\end{lem}
\begin{proof}
For any distinct $i,j,k\in\{1,\cdots,n\}$, we define
\begin{equation}\aligned
\phi&=\lambda_i^2+\lambda_j^2+\lambda_k^2-2\lambda_i\lambda_j-2\lambda_j\lambda_k-2\lambda_i\lambda_k,\\
\psi&=\lambda_j^2-4\lambda_i\lambda_j.\nonumber
 \endaligned\end{equation}
Firstly, let us estimate $\phi$.  Without loss of generality, we suppose
$$\lambda_i\lambda_j\leq\lambda_j\lambda_k\leq0,\;
\lambda_i\lambda_k\geq0.$$ Define
$$\lambda_i=-x\lambda_j,\lambda_k=-y\lambda_j, \; x\geq
y\geq0.$$ Now,
\begin{equation}\label{rl.17}
\phi=\lambda_i^2+\lambda_j^2+\lambda_k^2+2(x+y-xy)\lambda_j^2\leq
S+4+2(x\lambda_j^2-1+(1-x)y\lambda_j^2-1).
\end{equation}
Let
$$a=x\lambda_j^2-1, \; b=y\lambda_j^2-1,\; c=(1-x)y\lambda_j^2-1,$$
then (\ref{rl.17}) becomes
\begin{equation}\label{rl.18}
\phi\le S+4+2(a+c).
\end{equation} Noting $S\leq\frac{16}{15}n$ and
$S\geq\lambda_j^2+\frac{1}{n-1}(\sum_{k\neq j}\lambda_k)^2$, we
deduce
\begin{equation}
\lambda_j^2\leq\frac{16}{15}(n-1).
\end{equation}
In the case of $c=(1-x)y\lambda_j^2-1\geq0,$ which implies $x\le
1$ and $a,b\geq0$.
By Cauchy inequality and (3.10),
$$\aligned
c\leq&
\left(x(1-x)-\frac{15}{16(n-1)}\right)\lambda_j^2\leq\frac{4n-19}{32n-17}(x+1)^2\lambda_j^2
\leq\left(1-\frac{16}{5n}\right)\frac{2n-2}{16n-1}(x+1)^2\lambda_j^2,\\
&a\leq
\left(x-\frac{15}{16(n-1)}\right)\lambda_j^2\leq\frac{4n-4}{16n-1}(x+1)^2\lambda_j^2,\\
&b\leq
\left(y-\frac{15}{16(n-1)}\right)\lambda_j^2\leq\frac{4n-4}{16n-1}(y+1)^2\lambda_j^2.
\endaligned$$
For some $\epsilon>0$ to be defined later,
\begin{equation}\aligned
(a+c)^3=&a^3+c^3+3(a^2c+ac^2)\leq a^3+c^3+3\left(a^2c+\frac{\epsilon}{2}a^2c+\frac{1}{2\epsilon}c^3\right)\\
\leq&a^3+b^3+3\frac{2n-2}{16n-1}\left[\left(1+\frac{\epsilon}{2}\right)a^2\left(1-\frac{16}{5n}\right)(x+1)^2
+\frac{1}{\epsilon}b^2(1+y)^2\right]\lambda_j^2.\\
\endaligned\nonumber
\end{equation}
By the definition of $\sr{F},$ we have
\begin{equation}\aligned
\sr{F}\geq
2(\lambda_i-\lambda_j)^2(\lambda_i\lambda_j+1)^2&+2(\lambda_j-\lambda_k)^2(\lambda_j\lambda_k+1)^2\\
&=2(x+1)^2\lambda_j^2a^2+2(y+1)^2\lambda_j^2b^2.
\endaligned\nonumber
\end{equation}
Let $\epsilon=\sqrt{\frac{15n-16}{5n-16}}-1$, then
\begin{equation}\aligned\label{rl.19}
(a+c)^3\leq&
a^3+|b|^3+3\left(\sqrt{\frac{15n-16}{5n-16}}-1\right)^{-1}\frac{2n-2}{16n-1}(a^2(x+1)^2
+b^2(y+1)^2)\lambda_j^2\\
\leq& \left(2+3\left(\sqrt{\frac{15n-16}{5n-16}}-1\right)^{-1}\right)\frac{2n-2}{16n-1}(a^2(x+1)^2+b^2(y+1)^2)\lambda_j^2\\
\leq& \left(2+3\left(\sqrt{\frac{15n-16}{5n-16}}-1\right)^{-1}\right)\frac{n-1}{16n-1}\sr{F}.\\
\endaligned
\end{equation}
If $c\leq0$ and $a\geq0$, then (\ref{rl.19}) holds
clearly. Combining (\ref{rl.18}) and (\ref{rl.19}), we have the following
estimate
\begin{equation}\aligned\label{rl.20}
\phi\leq
S+4+\left(\left(2+3\left(\sqrt{\frac{15n-16}{5n-16}}-1\right)^{-1}\right)\frac{8(n-1)}{16n-1}\sr{F}\right)^{\f{1}{3}}.
\endaligned
\end{equation}
If $a\le0$, then $c\le0$ and the above inequality holds clearly. Hence (3.11) holds which is independent of the sign of $a,b,c$.

Secondly, let's estimate $\psi=\lambda_j^2-4\lambda_i\lambda_j$. 
In the case of $\psi-S-4>0$, then there is a $t>0$ such that $\lambda_i=-t\lambda_j$. Since
$$S\geq\lambda_i^2+\lambda_j^2+\frac{1}{n-2}(\sum_{k\neq
i,j}\lambda_k)^2=\frac{n-1}{n-2}\lambda_i^2+\frac{n-1}{n-2}\lambda_j^2+\frac{2}{n-2}\lambda_i\lambda_j,$$
then
\begin{equation}\aligned\label{rl.21}
\psi\leq& S-4\lambda_i\lambda_j-\frac{n-1}{n-2}\lambda_i^2-\frac{2}{n-2}\lambda_i\lambda_j-\frac{1}{n-2}\lambda_j^2\\
=&S+\left(-\frac{n-1}{n-2}t^2+\frac{4n-6}{n-2}t-\frac{1}{n-2}\right)\lambda_j^2.
\endaligned
\end{equation}
Since $n\geq6$ and (3.10), we have
\begin{equation}\aligned\label{rl.22}
\psi\leq&S+4+\left(-\frac{n-1}{n-2}t^2+\frac{4n-6}{n-2}t-\frac{1}{n-2}\right)\lambda_j^2-\frac{15}{4(n-1)}\lambda_j^2\\
\leq&S+4+\left(-\frac{n-1}{n-2}t^2+\frac{4n-6}{n-2}t-\frac{4}{n-2}\right)\lambda_j^2.
\endaligned
\end{equation}
By Cauchy inequality,
$$-\frac{n-1}{n-2}t^2+\frac{4n-6}{n-2}t-\frac{4}{n-2}\leq(t-\frac{12}{13(n-2)})(4-t).$$
By (\ref{rl.21}), $$\psi\leq
S-4\lambda_i\lambda_j-\lambda_i^2=S+(4t-t^2)\lambda_j^2,$$
combining (\ref{rl.22}), we have
\begin{equation}\aligned\label{rl.23}
(\psi-S-4)^3\leq&((4t-t^2)\lambda_j^2-4)^2\left(-\frac{n-1}{n-2}t^2+\frac{4n-6}{n-2}t-\frac{4}{n-2}\right)\lambda_j^2\\
\leq&((4t-t^2)\lambda_j^2-(4-t))^2\left(t-\frac{12}{13(n-2)}\right)(4-t)\lambda_j^2\\
=&\left(t-\frac{12}{13(n-2)}\right)(4-t)^3(t\lambda_j^2-1)^2\lambda_j^2.\\
\endaligned
\end{equation}
Now we define an auxiliary function
$$\omega(t,\xi)=\left(t-\frac{12}{13(n-2)}\right)(4-t)^3-\xi(1+t)^2.$$ Then
there exists the smallest $\xi$ such that
$$\sup_t\omega(t,\xi)=0.$$ For any $t_0$ satisfying
$\partial_t\omega(t_0,\xi)=\omega(t_0,\xi)=0,$ we solve  the
equations to get $$t_0=\sqrt{6+54p+9p^2}-2+3p,$$
$$\xi=\frac{1}{1+t_0}\left(2-2t_0+\frac{18}{13(n-2)}\right)(4-t_0)^2,$$ here
$p=\frac{1}{13(n-2)}.$ Since
$$t_0\geq\sqrt{6}+10p-2+3p=\sqrt{6}-2+13p,$$ then $$\xi\leq2\
\frac{3-\sqrt{6}-4p}{\sqrt{6}-1+13p}(6-\sqrt{6}-13p)^2,$$
Hence
\begin{equation}
\left(t-\frac{12}{13(n-2)}\right)(4-t)^3\leq2\
\frac{3-\sqrt{6}-4p}{\sqrt{6}-1+13p}(6-\sqrt{6}-13p)^2(1+t)^2.
\end{equation}
Noting
$$\sr{F}\geq2(\lambda_i-\lambda_j)^2(1+\lambda_i\lambda_j)^2=2(t+1)^2\lambda_j^2(t\lambda_j^2-1)^2$$
and (\ref{rl.23}), (3.16), we have
\begin{equation}\aligned\label{rl.24}
(\psi-S-4)^3\leq\frac{3-\sqrt{6}-4p}{\sqrt{6}-1+13p}(6-\sqrt{6}-13p)^2\sr{F}.
\endaligned
\end{equation}
If $\psi-S-4\leq0$, the above inequality holds clearly.
Let
$$C_3(n)=\left(\frac{3-\sqrt{6}-4p}{\sqrt{6}-1+13p}(6-\sqrt{6}-13p)^2\right)^{\f{1}{3}}.$$
By a calculation
$C_3(n)^3\geq\left(2+3\left(\sqrt{\frac{15n-16}{5n-16}}-1\right)^{-1}\right)\frac{8(n-1)}{16n-1}$
for $n\geq6.$ In fact,  both sides of the above inequality are 
increase in $n,$ we only need to check  the case $n=6$ and 
$$C_3(7)^3\geq \frac{7+3\sqrt{3}}{4} 
=\lim_{n\rightarrow\infty}\left(2+3\left(\sqrt{\frac{15n-16}{5n-16}}-1\right)^{-1}\right)\frac{8(n-1)}{16n-1}.$$

Combining (\ref{rl.20}) and (\ref{rl.24}), we
finally obtain
\begin{equation*}\aligned
3(\sr{A}-2\sr{B})\leq &\sum_{i,j,k\
distinct}h_{ijk}^2(\lambda_i^2+\lambda_j^2+\lambda_k^2-2\lambda_i\lambda_j
-2\lambda_j\lambda_k-2\lambda_i\lambda_k)\\
&+3\sum_{j,i\neq j}h_{iij}^2(\lambda_j^2-4\lambda_i\lambda_j)\\
\leq&\sum_{i,j,k\ distinct}h_{ijk}^2(S+4+C_3(n)\sr{F}^{1/3})+3\sum_{j,i\neq j}h_{iij}^2(S+4+C_3(n)\sr{F}^{1/3})\\
\leq&(S+4)|\nabla B|^2+C_3(n)|\nabla B|^2\sr{F}^{1/3}.
\endaligned
\end{equation*}
\end{proof}

\Section{Proof of Theorems}{Proof of Theorems}

\subsection{Proof of Theorem \ref{th1}}

\medskip
Since we already have known result for lower dimension, we assume the dimension $n\ge 4$.
By (\ref{bs}), (\ref{ci}) and (\ref{pt3}), we have
\begin{equation*}\aligned
\int_M\sum_{i,j,k,l}h_{ijkl}^2&\geq\frac{3}{2}\int_M(Sf_{ 4}-f_{ 3}^{ 2}-S^2-S(S-n))+\int_M\frac{3S(S-n)^2}{2(n+4)}\\
&=\frac{3}{2}\int_M(Sf_{ 4}-f_{ 3}^{ 2}-S^2)-\frac{3}{2}\int_M|\nabla B|^2+\int_M\frac{3S(S-n)^2}{2(n+4)}\\
&=\frac{3}{2}\int_M(\sr{A}-2\sr{B})+\frac{3}{8}\int_M|\nabla
S|^2-\frac{3}{2}\int_M|\nabla B|^2+\int_M\frac{3S(S-n)^2}{2(n+4)}.
\endaligned
\end{equation*}
Combining (\ref{ci}), for some fixed $0<\theta<1$ to be defined
later, we have
\begin{equation}\label{t1.1}\aligned
\frac{3\theta}{2}\int_M&(\sr{A}-2\sr{B})+\frac{3\theta}{8}\int_M|\nabla S|^2+\frac{3}{4}(1-\theta)\int_M\sr{F}+\int_M\frac{3S(S-n)^2}{2(n+4)}\\
\leq&\frac{3\theta}{2}\int_M|\nabla B|^2+\int_M|\nabla^2B|^2.
\endaligned
\end{equation}
Together with (\ref{pt1}), (\ref{t1.1}) and Lemma \ref{ml}, we
obtain
\begin{equation}\label{t1.2}\aligned
\frac{3}{4}&(1-\theta)\int_M\sr{F}+\int_M\frac{3S(S-n)^2}{2(n+4)}-\left(\frac{3}{2}
-\frac{3\theta}{8}\right)\int_M|\nabla S|^2\\
\leq&\int_M\left(S-2n-3+\frac{3\theta}{2}\right)|\nabla B|^2+\left(3-\frac{3\theta}{2}\right)\int_M(\sr{A}-2\sr{B})\\
\leq&\int_M\left(S-2n-3+\frac{3\theta}{2}\right)|\nabla B|^2
+\left(1-\frac{\theta}{2}\right)\int_M(2S|\nabla B|^2+C_1|\nabla B|^2f^{\frac{1}{3}})\\
\leq&\int_M\left((3-\theta)S-2n-3+\frac{3\theta}{2}\right)|\nabla B|^2+\frac{3}{4}(1-\theta)\int_M\sr{F}\\
&\qquad
+\frac{4}{9}C_1^\frac{3}{2}\left(1-\frac{\theta}{2}\right)^\frac{3}{2}(1-\theta)^{-\frac{1}{2}}\int_M|\nabla
B|^3,
\endaligned
\end{equation}
where we have used Young's inequality in the last step of
(\ref{t1.2}), then
\begin{equation}\label{t1.3}\aligned
\int_M\frac{3S(S-n)^2}{2(n+4)}\leq&\int_M\left((3-\theta)S-2n-3+\frac{3\theta}{2}\right)|\nabla
B|^2\\
&\qquad+\left(\frac{3}{2}-\frac{3\theta}{8}\right)\int_M|\nabla
S|^2+C_2(n,\theta)\int_M|\nabla B|^3,
\endaligned
\end{equation}
 where $C_2(n,\theta)=\frac{4}{9}C_1^\frac{3}{2}(1-\frac{\theta}{2})^\frac{3}{2}(1-\theta)^{-\frac{1}{2}}.$

By (\ref{bs}), for some $\epsilon>0$ to be defined later, we have
\begin{equation}\label{t1.4}\aligned
\int_M|\nabla B|^3=&\int_MS(S-n)|\nabla B|+\frac{1}{2}\int_M|\nabla B|\Delta S\\
=&\int_MS(S-n)|\nabla B|-\frac{1}{2}\int_M\nabla|\nabla B|\cdot\nabla S\\
\le&\int_MS(S-n)|\nabla B|+\epsilon\int_M|\nabla^2
B|^2+\frac{1}{16\epsilon}\int_M |\nabla S|^2.
\endaligned
\end{equation}
Combining (\ref{pt1}) and (\ref{pt2}), we obtain
\begin{equation*}\aligned
\int_M|\nabla^2B|^2\leq\int_M((\alpha+1)S-2n-3)|\nabla
B|^2+\frac{3}{2}\int_M|\nabla S|^2.
\endaligned
\end{equation*}
With the help of the above inequality, (\ref{t1.4}) becomes
\begin{equation}\aligned\label{t1.5}
\int_M|\nabla B|^3&\leq\int_MS(S-n)|\nabla
B|\\&+\int_M\epsilon((\alpha+1)S-2n-3)|\nabla B|^2
+\left(\frac{3\epsilon}{2}+\frac{1}{16\epsilon}\right)\int_M
|\nabla S|^2.
\endaligned
\end{equation}
Multiplying $S$ on the both sides of (\ref{bs}), and integrating
by parts, we see
\begin{equation}\aligned\label{t1.6}
\frac{1}{2}\int_M|\nabla S|^2=&\int_MS^2(S-n)-\int_MS|\nabla B|^2\\
=&\int_MS(S-n)^2+n\int_MS(S-n)-\int_MS|\nabla B|^2\\
=&\int_M(n-S)|\nabla B|^2+\int_MS(S-n)^2.
\endaligned
\end{equation}
Combining (\ref{t1.3}), (\ref{t1.5}) and (\ref{t1.6}), we get
\begin{equation}\aligned\label{t1.7}
0\leq&\int_M\left((3-\theta)S-2n-3+\frac{3\theta}{2}+C_2\epsilon((\alpha+1)S-2n-3)\right)|\nabla
B|^2\\
&+C_2\int_MS(S-n)|\nabla B|
+\left(\frac{3}{2}-\frac{3\theta}{8}+C_2\left(\frac{3\epsilon}{2}+\frac{1}{16\epsilon}\right)\right)
\int_M|\nabla S|^2\\
&\hskip1in-\int_M\frac{3S(S-n)^2}{2(n+4)}\\
\leq&\int_M\left((3-\theta)S-2n-3+\frac{3\theta}{2}+C_2\epsilon((\alpha+1)S-2n-3)\right)|\nabla
B|^2\\
&+C_2\int_MS(S-n)|\nabla B|
-\int_M(S-n)\left(3-\frac{3\theta}{4}+C_2(3\epsilon+\frac{1}{8\epsilon})\right)|\nabla B|^2\\
&\qquad +\left(3-\frac{3\theta}{4}+C_2(3\epsilon+\frac{1}{8\epsilon})-\frac{3}{2(n+4)}\right)\int_MS(S-n)^2 \\
=&\int_M\Big((1-\theta)n-3+\frac{3\theta}{2}+C_2\epsilon(\alpha
n-n-3)\\
&\hskip1in-(S-n)(\frac{\theta}{4}+C_2\epsilon(2-\alpha)+\frac{C_2}{8\epsilon})\Big)|\nabla B|^2\\
&+\left(3-\frac{3\theta}{4}+C_2\left(3\epsilon+\frac{1}{8\epsilon}\right)-\frac{3}{2(n+4)}\right)
\int_MS(S-n)^2\\
&\hskip1in+C_2\int_MS(S-n)|\nabla B|.
\endaligned
\end{equation}
By the assumption $n\leq S\leq n+\delta(n)$,  Cauchy-Schwartz
inequality and (\ref{bs}), we have
\begin{equation}\aligned\label{t1.8}
\int_M&S(S-n)|\nabla B|\leq2(n+\delta)\epsilon\int_MS(S-n)+\frac{1}{8(n+\delta)\epsilon}\int_MS(S-n)|\nabla B|^2\\
=&\int_M\left(2(n+\delta)\epsilon+\frac{S(S-n)}{8(n+\delta)\epsilon}\right)|\nabla
B|^2\leq\int_M\left(2(n+\delta)\epsilon+\frac{S-n}{8\epsilon}\right)|\nabla
B|^2.
\endaligned
\end{equation}

From (\ref{bs}), (\ref{t1.7}) and (\ref{t1.8}) we see that
\begin{equation}\label{t1.9}
0\le \int_M\left((1-\theta)n-3+\frac{3\theta}{2}+O(\ep)\right)|\n
B|^2,
\end{equation}
where we choose $\de=\ep^2.$ We could choose $\th$ close to $1$,
then it is easily seen that there exists $\ep>0,$ such that the
coefficient  of the integral in (\ref{t1.9}) is negative. This
forces $|\n B|=0.$ We now complete the proof of Theorem
\ref{th1}.

\subsection{Proof of Theorem \ref{th3}}
We assume $n\ge 6$.
In the proof of Theorem \ref{th1}, replacing Lemma \ref{ml} by
Lemma \ref{rl4} in (\ref{t1.2}), we have
\begin{equation}\aligned\label{t3.1}
&\frac{3}{4}(1-\theta)\int_M\sr{F}+\int_M\frac{3S(S-n)^2}{2(n+4)}-\left(\frac{3}{2}-\frac{3\theta}{8}\right)
\int_M|\nabla S|^2\\
\leq&\int_M(S-2n-3+\frac{3\theta}{2})|\nabla B|^2
+(1-\frac{\theta}{2})\int_M((S+4)|\nabla B|^2+C_3|\nabla B|^2\sr{F}^{\frac{1}{3}})\\
\leq&\int_M\left((2-\frac{\theta}{2})S-2n+1-\frac{\theta}{2}\right)|\nabla B|^2+\frac{3}{4}(1-\theta)\int_M\sr{F}\\
&\qquad
+\frac{4}{9}C_3^\frac{3}{2}\left(1-\frac{\theta}{2}\right)^\frac{3}{2}(1-\theta)^{-\frac{1}{2}}\int_M|\nabla
B|^3.
\endaligned
\end{equation}
Combining (\ref{t1.5}) and (\ref{t1.6}) we see
\begin{equation}\aligned\label{t3.2}
0\leq&\int_M\Big[-\frac{\theta}{2}(n+1)+1+C_4\epsilon(\alpha n-n-3)\\
&\qquad-(S-n)(1-\frac{\theta}{4}+C_4\epsilon(2-\alpha)+\frac{C_4}{8\epsilon})\Big]|\nabla B|^2\\
&+\left(3-\frac{3\theta}{4}+C_4(3\epsilon+\frac{1}{8\epsilon})-\frac{3}{2(n+4)}\right)
\int_MS(S-n)^2+C_4\int_MS(S-n)|\nabla B|,
\endaligned
\end{equation}
where
$C_4=C_4(n,\theta)=\frac{4}{9}C_3^\frac{3}{2}(1-\frac{\theta}{2})^\frac{3}{2}(1-\theta)^{-\frac{1}{2}}.$

Assume $n\leq S\leq n+\delta(n)$, by (\ref{t1.8}), we have
\begin{equation}\aligned\label{t3.3}
0\leq&\int_M\Big[-\frac{\theta}{2}(n+1)+1+C_4\epsilon(\alpha
n+n-3+2\delta)\\
&\qquad\qquad-(S-n)(1-\frac{\theta}{4}+C_4\epsilon(2-\alpha))\Big]|\nabla B|^2\\
&\quad+\left(3-\frac{3\theta}{4}+C_4(3\epsilon+\frac{1}{8\epsilon})-\frac{3}{2(n+4)}\right)\int_MS(S-n)^2\\
\leq&\int_M\Big[-\frac{\theta}{2}(n+1)+1+C_4\epsilon(\alpha n+n-3+2\delta)\\
&\qquad\qquad-(S-n)(1-\frac{\theta}{4}+C_4\epsilon(2-\alpha))\Big]|\nabla B|^2\\
&\qquad+(3-\frac{3\theta}{4}+C_4(3\epsilon+\frac{1}{8\epsilon})-\frac{3}{2(n+4)})\delta\int_M|\nabla B|^2\\
=&\Big(-\frac{\theta}{2}(n+1)+1+C_4\epsilon(\alpha n+n-3+5\delta)+\frac{C_4}{8\epsilon}\delta\\
&\qquad+(\frac{3(2n+5)}{2(n+4)}-\frac{3\theta}{4})\delta\Big)\int_M|\nabla B|^2\\
&\qquad-\int_M\left(1-\frac{\theta}{4}+C_4\epsilon(2-\alpha)\right)(S-n)|\nabla
B|^2.
\endaligned
\end{equation}

Let $\epsilon=\sqrt{\frac{\delta}{8(\alpha n+n-3+5\delta)}}$ and
$\theta=0.84$, then
$$C_4(n)=\frac{4}{9}\times0.58^{3/2}\times0.16^{-1/2}\times\sqrt{\frac{3-\sqrt{6}-4p}{\sqrt{6}-1+13p}}(6-\sqrt{6}-13p),$$
where $p=\frac{1}{13(n-2)}.$ We have
$C_4(n)\leq\lim_{l\rightarrow\infty}C_4(l)\leq1.1.$ Combining
$\delta(n)\leq\frac{n}{15}$ and $\a=\frac{\sqrt{17}+1}{2}$ we obtain
$0.79+C_4\epsilon(2-\alpha)\geq0.$ From (\ref{t3.3}) we get
\begin{equation}\aligned\label{t3.4}
0\leq\left(-0.42n+0.58+C_4\sqrt{\frac{\delta}{2}(\alpha
n+n-3+5\delta)}+\left(\frac{3(2n+5)}{2(n+4)}-0.63\right)\delta\right)\int_M|\nabla
B|^2.
\endaligned
\end{equation}
If $\delta(n)=\f{n}{23},$ then the coefficient of the integral in
(\ref{t3.4}) is negative, hence,\\  $|\nabla B|\equiv 0,\,
S\equiv n.$ The proof is complete.

\bibliographystyle{amsplain}

\end{document}